\documentclass[11pt,leqno]{amsart}
\usepackage[utf8]{inputenc}

\usepackage{amsmath}
\usepackage{amsthm}
\usepackage{enumerate}
\usepackage{mathtools}
 
\usepackage[colorlinks=false]{hyperref}

\usepackage[charter]{mathdesign}

\usepackage{xy}
\xyoption{all}

\numberwithin{equation}{section}

\theoremstyle{plain}
    \newtheorem{theorem}[equation]{Theorem}
    \newtheorem{lemma}[equation]{Lemma}

    \newtheorem*{theorem*}{Theorem}
    \newtheorem*{proposition*}{Proposition}
    \newtheorem*{corollary*}{Corollary}
    \newtheorem*{lemma*}{Lemma}
    \newtheorem*{conjecture*}{Conjecture}
    \newtheorem{definition-theorem}[equation]{Definition/Theorem}
    \newtheorem{definition-lemma}[equation]{Definition/Lemma}
\theoremstyle{definition}
    \newtheorem{definition}[equation]{Definition}
    \newtheorem{example}[equation]{Example}
    \newtheorem{examples}[equation]{Examples}

    \newcommand{\C}{\mathbb{C}}
    \newcommand{\N}{\mathbb{N}}

   	\renewcommand{\phi}{\varphi}
	\let\epsilon\varepsilon

    \newcommand{\CB}{\operatorname{CB}}
    \newcommand{\Compact}{\operatorname{K}}
    \newcommand{\Adjointable}{\operatorname{L}}
    \newcommand{\Multiplier}{\operatorname{M}}
    \newcommand{\Unitary}{\operatorname{U}}
\newcommand{\into}{\hookrightarrow}
\newcommand{\restrict}{\raisebox{-.5ex}{$|$}}

\newcommand{\id}{\mathrm{id}}
\DeclareMathOperator{\Ind}{Ind}
\DeclareMathOperator{\lspan}{span}
\DeclareMathOperator{\image}{image}

 \newcommand{\h}{\mathrm{h}}
\newcommand{\bra}[1]{{\langle{#1}\vert}}
\newcommand{\ket}[1]{{\vert {#1}\rangle}}
\newcommand{\TC}{\operatorname{L}^1}

\newcommand{\trace}{\operatorname{trace}}

\title{Trace-class operators on Hilbert modules and Haagerup tensor products}

\author{Tyrone Crisp}

\author{Michael Rosbotham}
\address{Department of Mathematics \& Statistics, University of Maine.
5752 Neville Hall, Room 237.
Orono, ME 04469 USA}
\email{tyrone.crisp@maine.edu \textrm{(corresponding author)}}
\email{michael.rosbotham@maine.edu}


 
\begin{document}

\begin{abstract}
We show that the space of trace-class operators on a Hilbert module over a commutative $C^*$-algebra, as defined and studied in earlier work of Stern and van Suijlekom (Journal of Functional Analysis, 2021), is completely isometrically isomorphic to  a Haagerup tensor product of the module with its operator-theoretic adjoint. This generalises a well-known property of Hilbert spaces. In the course of proving this, we also obtain a new proof of a result of Stern-van Suijlekom concerning the equivalence between two definitions of trace-class operators on Hilbert modules.
\end{abstract}

\maketitle

\section{Introduction}

The study of continuous families of Hilbert spaces and of Hilbert-space operators leads naturally to the notion of Hilbert modules over commutative $C^*$-algebras, and operators on these modules \cite{Kaplansky,Paschke,Takahashi}. In \cite{SvS} Stern and van Suijlekom defined and studied Schatten classes of operators on Hilbert modules over a commutative $C^*$-algebra $A$, giving two equivalent characterisations of these operators: one in terms of an $A$-valued trace, and the other in terms of the family of $\C$-valued traces arising from localisation at each point of the spectrum $\widehat{A}$.

In this paper we focus on  the space $\TC_A(F)$ of trace-class operators on a Hilbert module $F$ over a commutative $C^*$-algebra $A$. We first observe that the definition of trace-class operators given in \cite{SvS} can be extended by using frames of multipliers, as introduced by Raeburn and Thompson in \cite{RT}. This modified definition applies to some situations where that of \cite{SvS} does not apply (eg, to modules of the form $A^k$ where $A$ is a non-$\sigma$-unital $C^*$-algebra). In situations where both definitions apply they agree, and the extra flexibility afforded by frames of multipliers is sometimes useful in simplifying computations.

The main result of this paper concerns the connection between the space of trace-class operators $\TC_A(F)$ and the \emph{Haagerup tensor product} $\otimes^{\h}$ from operator-space theory \cite{Effros-Kishimoto}. It is known (\cite{Blecher-Paulsen-tensor,Effros-Ruan-self-duality}) that if $H$ is a separable Hilbert space then the map
\[
\phi_H : H^* \otimes^{\h} H \to \TC(H),\qquad \bra{\xi}\otimes \ket{\eta}\mapsto \ket{\eta}\bra{\xi}
\]
is isometric, and indeed completely isometric when we equip $\TC(H)$ with the operator-space structure coming from realising $\TC(H)$ as the dual of $\Compact(H)$. Our main result extends this isomorphism to Hilbert modules over an arbitrary commutative $C^*$-algebra $A$: we prove  in Theorem \ref{thm:ci} that if $F$ is such a module, countably generated by multipliers, then the map
\[
\phi_F : F^* \otimes_A^{\h} F \to \TC_A(F),\qquad \bra{\xi}\otimes \ket{\eta}\mapsto \ket{\eta}\bra{\xi}
\]
is a completely isometric isomorphism, when $\TC_A(F)$ is given a natural operator-space structure, and where $F^*$ denotes the operator-theoretic adjoint of $F$, i.e., $F^* = \Compact_A(F,A)$. In the course of proving this result, we also obtain a new proof of the equivalence between the two definitions of trace-class operators established in \cite{SvS}.

Our result complements earlier work on Haagerup tensor products of Hilbert modules, primarily due to Blecher (eg, \cite{Blecher-newapproach}). Blecher proved, among other things, that if $F$ is a Hilbert module over any $C^*$-algebra $A$ (not necessarily commutative), then $F\otimes^{\h}_A F^*$ is isomorphic to the $C^*$-algebra $\Compact_A(F)$ of $A$-compact operators on $F$. The restriction to commutative $C^*$-algebras in our isomorphism $F^*\otimes^{\h}_A F \cong \TC_A(F)$ appears to us to be essential, because of the difficulties associated with defining an $A$-valued trace when $A$ is not commutative.

Haagerup tensor products of the form $F^*\otimes^{\h}_A F$ play a central role in the descent theory developed in \cite{Crisp-descent}. A major challenge in applying that theory is to compute these tensor products in useful, concrete terms, and the results of this paper show how this can be done in the special case where $A$ is commutative and $F$ is a right Hilbert $A$-module. We present an example of this kind, related to unitary group representations, at the end of this paper, in Section \ref{sec:example}.  See  \cite{Crisp-gluing} and \cite{Crisp-basechange} for computations of $F^*\otimes^{\h}_A F$ in other settings.

The Haagerup tensor product belongs to operator-space theory, but its usefulness is not confined to purely operator-space-theoretic applications. To emphasise this point, and to make our paper more accessible to readers interested in Hilbert modules but not necessarily well-versed in operator-space theory, we have divided the main argument into  two sections. Section \ref{sec:free-modules}, which contains our  proof of (part of) \cite[Theorem 3.18]{SvS}, does not require any operator-space background on the part of the reader. In Section \ref{sec:oss}, which concerns the operator-space structure on the space of trace-class operators, we refer more freely to the literature on operator spaces. Before that, in Section \ref{sec:preliminaries}, we review some essential background from \cite{SvS} and \cite{RT}, and indicate how to use the technology of the latter to extend the reach of the former.

\section{Frames of multipliers and trace-class operators}\label{sec:preliminaries}

The purpose of this section is to recall some background and establish notation, and to point out that the notion of trace-class operators studied in \cite{SvS} admits an easy and useful generalisation using frames of multipliers \cite{RT}.

We will assume that the reader is familiar with the basic theory of Hilbert modules, as explained in \cite{Lance} for exampe; see also  \cite{Blecher-newapproach} for a presentation of much of the basic theory from a point of view closely aligned with the one taken here. 

Let $A$ be a commutative $C^*$-algebra. In this paper all Hilbert modules are right modules, and all $A$-valued inner products $\langle\ |\ \rangle$ are $A$-linear in their right-hand argument. We write $\Adjointable_A$ and $\Compact_A$ for the spaces of adjointable and of compact operators (respectively) between Hilbert $A$-modules, omitting the $A$ when $A=\C$. If $F$ is a Hilbert $A$-module then for each $\xi\in F$ we have operators $\bra{\xi}\in \Compact_A(F,A)$ and $\ket{\xi}\in \Compact_A(A,F)$, defined by $\bra{\xi}:\eta\mapsto \langle \xi\,|\,\eta\rangle$ and $\ket{\xi}:a\mapsto \xi a$. The map $\xi\mapsto \ket{\xi}$ is an isometric isomorphism from $F$ to $\Compact_A(A,F)$, and because of this we sometimes blur the distinction between $\xi$ and $\ket{\xi}$.

\subsection*{Localisation:} For each Hilbert $A$-module $F$, and for each point $x$ in the spectrum $\widehat{A}$, the quotient  of $F$ by its closed submodule $\{\xi  \in F\ |\ \langle\xi\, |\, \xi\rangle(x)=0\}$ is a Hilbert space, which we denote by $F_x$. The image of $\xi\in F$ in  $F_x$ is denoted $\xi_x$, and the $\C$-valued inner product on $F_x$ is $\langle \xi_x\, |\, \eta_x\rangle_{F_x}\coloneqq \langle \xi\, |\, \eta\rangle_F(x)$.   This \emph{localisation} procedure is a $*$-functor: each adjointable operator $t\in \Adjointable_A(F,E)$ induces, for each $x\in \widehat{A}$, a bounded operator $t_x : F_x\to E_x$, satisfying $t(\xi)_x = t_x (\xi_x)$ for all $\xi\in F$; and we have $(t\circ r)_x = t_x\circ r_x$ and $(t^*)_x = (t_x)^*$ for all adjointable operators $t$ and $r$. For each $t\in \Adjointable_A(F,E)$ the function $x\mapsto \|t_x\|$ is  bounded on $\widehat{A}$, and $\|t\|_{\Adjointable_A(F,E)}= \sup_{x\in \widehat{A}} \|t_x\|$. 

\begin{example}
Let $H$ be a separable Hilbert space. The space $C_0(\widehat{A},H)$ of continuous, vanishing-at-infinity, $H$-valued functions on the spectrum of $A$ is a Hilbert module over $A\cong C_0(\widehat{A})$: the module structure is by pointwise multiplication, and the $A$-valued inner product is given by $\langle f\, |\, g\rangle_{C_0(\widehat{A},H)}(x)\coloneqq \langle f(x)\, |\, g(x)\rangle_H$ (where $x\in \widehat{A}$). Localisation gives a $C^*$-algebra isomorphism  $\Compact_A(C_0(\widehat{A},H))\cong C_0(\widehat{A},\Compact(H))$.
\end{example}

\subsection*{Frames of multipliers:} We shall briefly recall some facts about multiplier frames for Hilbert modules as developed in \cite{RT}. We continue to assume that $A$ is a commutative $C^*$-algebra, although commutativity is not necessary for much of this section. We denote by $\Multiplier(A)$ the multiplier algebra of $A$.

If $F$ is a Hilbert $A$-module then we define $\Multiplier(F)\coloneqq \Adjointable_A(A,F)$, which is a Hilbert $\Multiplier(A)$-module with inner product $\langle r\,|\, s\rangle = r^*s \in \Adjointable_A(A)=\Multiplier(A)$. The module $F$ sits inside $\Multiplier(F)$ as the space of compact operators. Since the composition of an adjointable operator with a compact operator is compact,  for each $\xi\in F$ and each $\mu\in \Multiplier(F)$ we have $\langle \xi\, |\, \mu\rangle \in A$.

Each adjointable operator $t\in \Adjointable_A(F,E)$ extends to an adjointable operator $t\in \Adjointable_{\Multiplier(A)}(\Multiplier(F),\Multiplier(E))$, namely the operator $r\mapsto t\circ r$. This extension procedure is a $*$-functor. If $t$ is compact then  $t(\mu)\in E$ for every $\mu\in \Multiplier(F)$.

We say that a Hilbert $A$-module $F$ is \emph{countably generated by multipliers} if there is a countable subset $G\subseteq \Multiplier(F)$ such that $F=\overline{\lspan}\{ga\ |\ g\in G,\ a\in A\}$. For example, $A$ is countably generated by multipliers as a module over itself, since the single multiplier $\id_A$ suffices as a generator. 

\begin{definition}[\cite{RT}] 
A \emph{frame of multipliers} for $F$ is a sequence $(\beta_i)_{i\in \N}$ in $\Multiplier(F)$ such that for all $\xi,\eta\in F$ we have
\(
\langle \xi\, |\, \eta\rangle = \sum_{i=1}^\infty \langle \xi\, |\, \beta_i\rangle\langle \beta_i \, |\, \eta\rangle,
\)
where the sum is required to converge in the norm topology on $A$. A \emph{frame} for $F$ is a frame of multipliers $(\beta_i)$ with $\beta_i\in F$ for all $i$.
\end{definition}

Frames of multipliers exist more generally than do frames: for instance, while the $A$-module $A$ might not possess a frame,  the sequence $(\id_A,0,0,\ldots)$ is a frame of multipliers. Even in situations where frames do exist, it is sometimes more convenient to use frames of multipliers. For instance:

\begin{example}\label{example:free-module-frame}
Let $H$ be a separable Hilbert space, with orthonormal basis $(\epsilon_i)$, and consider the Hilbert $A$-module $C_0(\widehat{A},H)$. For each $i$ we consider the multiplier $\widetilde{\epsilon_i}\in \Adjointable_A(A,C_0(\widehat{A},H))$ defined by $\widetilde{\epsilon_i}(a)(x)\coloneqq a(x)\epsilon_i$. The sequence $(\widetilde{\epsilon_i})$ is then a frame of multipliers for $C_0(\widehat{A},H)$.

In the case of $H= \C^k$ we identify $C_0(\widehat{A},\C^k)\cong A^k$ in the obvious way. Taking $\epsilon_1,\ldots,\epsilon_k$ to be the standard basis for $\C^k$, we arrive at a finite frame of multipliers $(\widetilde{\epsilon_i})$ for $A^k$, where $\widetilde{\epsilon_i}:A\to A^k$ inserts $a$ into the $i$th coordinate (and leaves the other coordinates $0$). This frame of multipliers is a frame if and only if $A$ is unital.
\end{example}

The following theorem summarises the main results of \cite{RT}.

\begin{theorem}[{\cite{RT}}]\label{thm:frame-properties}
The following are equivalent for a Hilbert module $F$ over a commutative $C^*$-algebra $A$:
\begin{enumerate}[\rm(a)]
\item $F$ is countably generated by multipliers.
\item There is a separable Hilbert space $H$ and an adjointable map $\theta: F\to C_0(\widehat{A},H)$ with $\theta^*\theta=\id_F$.
\item  $F$ admits a frame of multipliers.
\end{enumerate}
Moreover, if $(\beta_i)$ is a frame of multipliers for $F$ then for each $\eta\in F$ the sum $\sum_{i=1}^\infty \ket{\beta_i}\langle \beta_i\, |\, \eta\rangle$ converges in norm to $\eta$.
\hfill\qed
\end{theorem}

The localisation procedure for adjointable operators applies in particular to elements of $\Multiplier(F)=\Adjointable_A(A,F)$: if $\mu\in \Multiplier(F)$ then for each $x\in \widehat{A}$ we have $\mu_x\in \Adjointable(\C, F_x)\cong F_x$. The inner-product formula $\langle \xi_x\, |\, \eta_x\rangle = \langle \xi\, |\, \eta\rangle(x)$ continues to hold when $\xi$ and $\eta$ are multipliers of $F$. Stern and van Suijlekom observed in \cite[Proposition 2.10]{SvS} that frames in Hilbert modules localise to give frames in Hilbert spaces, and the same is true, for the same reason, of frames of multipliers.

\subsection*{Trace-class operators on Hilbert modules:} We now recall the definition of trace-class operators from \cite{SvS}, extended by using frames of multipliers. Let $A$ be a commutative $C^*$-algebra, and let $F$ be a Hilbert $A$-module that is countably generated by multipliers.

\begin{definition}
For each frame of multipliers $\beta$ for $F$, and each positive operator $t\in \Adjointable_A(F)$, we define
\(
\trace_\beta(t) \coloneqq \sum_{i=1}^\infty \langle \beta_i\, |\, t\beta_i\rangle
\) 
if the series converges in the norm on $\Multiplier(A)$ \emph{to an element of $A$}; otherwise $\trace_{\beta}(t)$ is undefined.
\end{definition}

Note that the series defining $\trace_\beta(t)$ might well converge in $\Multiplier(A)$ to an element not in $A$, as the next examples make clear.

\begin{examples}\label{example:Ak-trace}
\begin{enumerate}[\rm(1)]
\item Consider the Hilbert $A$-module $A^k$, equipped with the standard frame of multipliers $\widetilde{\epsilon_1},\ldots,\widetilde{\epsilon_k}$ (see Example \ref{example:free-module-frame}). Each  positive $t\in \Adjointable_A(A^k)$ can be represented by a $k\times k$ matrix over $\Multiplier(A)$, and the sum defining $\trace_{\widetilde{\epsilon}}(t)$ is the sum of the diagonal entries of this matrix. This sum certainly exists in $\Multiplier(A)$, but if it does not lie in $A$ then $\trace_{\widetilde{\epsilon}}(t)$ is undefined.
\item Let $A=C_0(\mathbb{N})$ and $F=C_0(\mathbb{N},\ell^2)$, where $\ell^2=\ell^2(\N)$. Let $t\in \Compact_A(F)=C_0(\mathbb{N},\Compact(\ell^2))$ be the positive operator defined by $t(n) = \frac{1}{n} \sum_{i=1}^n \ket{\epsilon_i}\bra{\epsilon_i}$, where $(\epsilon_i)$ is the standard basis for $\ell^2$. We have for each $j,n\in \N$
\[
\langle \widetilde{\epsilon_j}\, |\, t\widetilde{\epsilon_j}\rangle(n) = \begin{cases} \frac{1}{n} & \text{if }j\leq n \\ 0 & \text{if }j>n.\end{cases}
\]
Thus all of the partial sums of $\trace_{\widetilde{\epsilon}}(t)$ lie in $A$, but the sum does not converge in norm---rather, it converges (to $1$) in the strict topology on $\Multiplier(A)$. So $\trace_{\widetilde{\epsilon}}(t)$ is undefined. Note that in this example we have $t(n)\in \TC(\ell^2)$ for every $n$, but the function $n\mapsto \trace(t(n))$ does not vanish at infinity.
\end{enumerate}
\end{examples}

\begin{theorem}[{\cite[Theorem 3.5]{SvS}}] \label{thm:tc-pointwise}
Let $t\in \Adjointable_A(F)$ be a positive adjointable operator. The following are equivalent:
\begin{enumerate}[\rm(a)]
\item $\trace_\beta(t)$ exists  in $A$, for some frame of multipliers $\beta$ for $F$.
\item $\trace_\beta(t)$ exists  in $A$, for every frame of multipliers $\beta$ for $F$.
\item for each $x\in \widehat{A}$ the operator $t_x\in \Adjointable(F_x)$ is of trace class, and the function $x\mapsto \trace(t_x)$ lies in $A$ (i.e., it is a $C_0$-function on $\widehat{A}$).
\end{enumerate}
If these equivalent conditions are satisfied, then for each frame of multipliers $\beta$ for $F$, and each $x\in \widehat{A}$, we have $\trace_\beta(t)(x) = \trace(t_x)$. In particular, $\trace_\beta(t)$ is independent of $\beta$.
\end{theorem}

\begin{proof}
Despite the extra generality coming from our use of frames of multipliers, the proof is identical to \cite[Theorem 3.5]{SvS}. The same argument goes through because localisation is still compatible with inner products, and frames of multipliers for Hilbert modules still give rise to frames for Hilbert spaces upon localisation. 
%
%
\end{proof}

In view of this result, we shall henceforth just write $\trace$ instead of $\trace_\beta$.

\begin{definition}\label{def:tc}
An operator $t\in \Adjointable_A(F)$ is of \emph{trace class} if $\trace(|t|)$ is defined in $A$. We let $\TC_A(F)$  denote the set of trace-class operators  on $F$.
\end{definition}

When $A=\C$, Hilbert $A$-modules are the same thing as Hilbert spaces, and the above definition of trace-class operators coincides with the usual one. More generally, if $F$ admits a frame then the space $\TC_A(F)$ defined above coincides with the one studied in \cite{SvS}. The difference between our set-up and that of \cite{SvS} is, firstly, that $\TC_A(F)$ is now also defined when $F$ has a frame of multipliers but not a frame; and secondly, that even for modules that do admit a frame, one can compute the trace using any frame of multipliers. This extra flexibility is sometimes useful, cf. Example \ref{example:Ak-trace}(1).


As pointed out in \cite{SvS}, it is not clear from Definition \ref{def:tc} alone that the set $\TC_A(F)$ has many nice properties; for example, it is not clear that this set is closed under addition. This and other properties of $\TC_A(F)$ will follow from the results of  the next two sections.

\section{Trace-class operators on $C_0(\widehat{A},H)$}\label{sec:free-modules}

In this section we use the Haagerup tensor product to give an alternative proof of the $p=1$ case of \cite[Theorem 3.18]{SvS}. Let $A$ be a commutative $C^*$-algebra, and let $H$ be a separable Hilbert space. The localisation procedure recalled in Section \ref{sec:preliminaries} associates, to each $t\in \Adjointable_A(C_0(\widehat{A},H))$, an operator-valued function $\widehat{A} \xrightarrow{x\mapsto t_x} \Adjointable(H)$. Theorem \ref{thm:tc-pointwise} shows that if the operator $t$ is of trace class, then the function $x\mapsto t_x$ lies in $C_0(\widehat{A},\TC(H))$, where the space $\TC(H)$ of trace-class operators on $H$ is given the norm $t\mapsto \trace(|t|)$. We shall prove the converse:

\begin{theorem}[cf. {\cite[Theorem 3.18]{SvS}}]\label{thm:SvS}
Let $H$ be a separable Hilbert space, let $A$ be a commutative $C^*$-algebra, and let $t\in \Adjointable_A(C_0(\widehat{A},H))$ be an adjointable operator. We have $t\in\TC_A(C_0(\widehat{A},H))$ if and only if the function $x\mapsto t_x$ lies in  $C_0(\widehat{A},\TC(H))$.
\end{theorem}

If $F$ is a Hilbert module over a $C^*$-algebra $A$ then we identify $F$ with $\Compact_A(A,F)$ via the map $\eta\mapsto \ket{\eta}$, as explained in Section \ref{sec:preliminaries}. We then let $F^*$ denote the operator-theoretic adjoint of this set of operators: $F^* \coloneqq \Compact_A(F,A) = \{ \bra{\xi}\ |\ \xi\in F\}$. We emphasise that in this paper ``$*$'' will always denote the adjoint of an operator or set of operators, and never the dual space.

\begin{definition}\label{def:Haagerup-norm}
The \emph{Haagerup norm} on the algebraic tensor product $F^*\otimes F$ is defined by
\[
\| u\|_{\h} = \inf\left\{ \big\|\textstyle\sum_{i=1}^k \langle \xi_i\, |\, \xi_i\rangle\big\|^{1/2}  \big\| \textstyle \sum_{i=1}^k \langle \eta_i\, |\, \eta_i\rangle \big\|^{1/2} \right\}
\]
where the infimum is taken over all the ways of writing  $u$ as a sum of elementary tensors $u = \sum_{i=1}^k \bra{\xi_i}\otimes \ket{\eta_i}$. The \emph{Haagerup tensor product} $F^*\otimes^{\h} F$ is the completion of the algebraic tensor product in this norm.
\end{definition}

See \cite[Chapter 9]{ER} or \cite[Chapter 5]{Pisier} for background on the Haagerup tensor product.

\begin{lemma}\label{lem:phiF}
The map $\phi_F:F^*\otimes F \to \Compact_A(F)$ defined by $\phi_F(\bra{\xi} \otimes \ket{\eta})=\ket{\eta}\bra{\xi}$ is a contraction for the Haagerup norm.
\end{lemma}

\begin{proof}
For all $\xi_1,\eta_1,\ldots,\xi_k,\eta_k\in F$ we have 
\(
\phi_F\left(\textstyle\sum_{i=1}^k \bra{\xi_i}\otimes \ket{\eta_i}\right) = \eta \xi,
\)
where
\[
\eta = \begin{bmatrix}  \ket{\eta_1} & \cdots & \ket{\eta_k} \end{bmatrix} \in \Compact_A(A^k, F) \quad \text{and}\quad \xi = \begin{bmatrix} \bra{\xi_1} \\ \vdots \\ \bra{\xi_k}\end{bmatrix} \in \Compact_A(F,A^k).
\]
Now, using the fact that the trace dominates the operator norm on positive Hilbert-space operators, we estimate
\begin{multline*}
\|\eta\|_{\Compact_A(A^k,F)}^2  = \| \eta \eta^*\|_{\Compact_A(F)}  = \big\| \textstyle\sum_i \ket{\eta_i}\bra{\eta_i} \big\|  = \sup_{x\in \widehat{A}} \big\| \sum_i \ket{\eta_{i,x}}\bra{\eta_{i,x}} \big\| \\
 \leq \textstyle \sup_{x\in \widehat{A}} \sum_i \langle \eta_{i,x}\, |\, \eta_{i,x} \rangle  =\left\|\sum_i \langle \eta_i\, |\, \eta_i\rangle\right\|_A.
\end{multline*}
A similar argument shows that $\|\xi\|^2_{\Compact_A(F,A^k)} \leq \| \sum_i \langle \xi_i\, |\, \xi_i\rangle \|_A$, and now the inequality of operator norms $\| \eta\xi\|\leq \|\eta\| \|\xi\|$ ensures that 
\[
\left\|\phi_F\left(\textstyle\sum_i \bra{\xi_i}\otimes \ket{\eta_i}\right)\right\|_{\Compact_A(F)} = \|\eta\xi\|_{\Compact_A(F)} \leq \big\|\textstyle\sum_i\langle \xi_i\, |\, \xi_i\rangle \big\|_A^{1/2} \big\| \textstyle\sum_i \langle \eta_i\, |\, \eta_i\rangle \big\|_A^{1/2}.
\]
Taking the infimum over $\xi$ and $\eta$  shows that $\phi_F$ is contractive.
\end{proof}

The map $\phi_F$ thus extends to a contraction $F^*\otimes^{\h} F\to \Compact_A(F)$, which we continue to denote by $\phi_F$.

We now focus for a moment on the special case $A=\C$. Classical operator theory (see, e.g. \cite[\S 18]{Conway} or \cite[Section 2.4]{Murphy}) tells us that if $H$ is a separable Hilbert space, then $\TC(H)$ is a linear subspace of $\Compact(H)$, and a Banach space under the norm $t\mapsto \trace(|t|)$.

\begin{lemma}\label{lem:phiH}
Let $H$ be a separable Hilbert space, regarded as a Hilbert module over $\C$. The map $\phi_H$ gives an isometric isomorphism $H^*\otimes^{\h} H \xrightarrow{\cong}\TC(H)$.
\end{lemma}

This result---and indeed, the stronger assertion that $\phi_H$ is a \emph{completely} isometric isomorphism---is well known in operator-space theory; see \cite[p.275]{Blecher-Paulsen-tensor} and \cite[Corollary 4.4(c)]{Effros-Ruan-self-duality}. We shall give a simple direct proof here, both to keep this part of our presentation  self-contained, and to illuminate some of the later arguments.

\begin{proof}
We begin by noting, as in the proof of Lemma \ref{lem:phiF}, that for $u=\sum_{i=1}^k \bra{\xi_i}\otimes \ket{\eta_i} \in H^*\otimes H$ we have $\phi_H(u) = \eta \xi$ where $\eta = \sum_{i=1}^k \ket{\eta_i}\bra{\epsilon_i}\in \Compact(\C^k,H)$ and $\xi = \sum_{i=1}^k \ket{\epsilon_i}\bra{\xi_i}\in \Compact(H,\C^k)$, with $(\epsilon_i)$ denoting the standard orthonormal basis for $\C^k$. The H\"older inequality gives
\[
\begin{aligned}
\trace(|\phi_H(u)|) = \trace(|\eta\xi|) & \leq \trace(\eta^*\eta)^{1/2}\trace(\xi^*\xi)^{1/2} \\
& =  \big(\textstyle \sum_{i=1}^k \langle \eta_i\, |\, \eta_i\rangle \big)^{1/2} \big(\textstyle\sum_{i=1}^k \langle \xi_i\, |\, \xi_i\rangle\big)^{1/2}.
\end{aligned}
\]
Taking the infimum over all choices of $\xi_i$ and $\eta_i$ shows that $\trace(|\phi_H(u)|)\leq \|u\|_{\h}$, and so 
the map $\phi_H$ is a contraction for the trace norm on $\TC(H)$. 

To see that $\phi_H$ is an isometry, take $u\in H^*\otimes H$ and let $t\coloneqq \phi_H(u)$, a finite-rank operator on $H$. Using the polar decomposition $t=v|t|$ we write $r=v|t|^{1/2}$ and $s=|t|^{1/2}$. Both $r$ and $s$ have finite rank, so we can find a finite orthonormal set $\{\epsilon_1,\ldots, \epsilon_k\}\subset H$ such that $r=\sum_{i=1}^k \ket{r\epsilon_i}\bra{\epsilon_i}$ and $s=s^*=\sum_{i=1}^k \ket{\epsilon_i}\bra{s\epsilon_i}$. We have
\[
\phi_H(u) = t = rs = \textstyle\sum_{i=1}^k \ket{r\epsilon_i}\bra{s\epsilon_i} = \phi_H\left( \textstyle \sum_{i=1}^k \bra{s\epsilon_i}\otimes \ket{r\epsilon_i}\right).
\]
The map $\phi_H$ is injective on the algebraic tensor product---indeed, $\phi_H$ restricts to an isomorphism between $H^*\otimes H$ and the space of finite-rank operators on $H$---so $u= \sum_{i=1}^k \bra{s\epsilon_i}\otimes \ket{r\epsilon_i}$. We therefore have
\[
\begin{aligned}
\|u\|_{\h} &\leq \left(\textstyle \sum_{i=1}^k \langle s\epsilon_i\, |\, s\epsilon_i\rangle \right)^{1/2}  \left( \textstyle \sum_{i=1}^k \langle r\epsilon_i\, |\, r\epsilon_i\rangle \right)^{1/2} \\
& = \trace(s^*s)^{1/2}\trace(r^*r)^{1/2} = \trace(|\phi_H(u)|),
\end{aligned}
\]
showing that $\phi_H$ is isometric for the trace norm on $\TC(H)$. The image of $\phi_H$ is obviously dense in $\TC(H)$, as it contains all of the finite-rank operators, and so $\phi_H$ is an isometric isomorphism.
\end{proof}

We now return to the case of a general commutative $C^*$-algebra $A$. 

\begin{lemma}\label{lem:image-phi-1}
If $H$ is a separable Hilbert space, then the image of $\phi_{C_0(\widehat{A},H)}$ contains $\TC_A(C_0(\widehat{A},H))$.
\end{lemma}

The proof uses a modification of the polar-decomposition technique from our proof of Lemma \ref{lem:phiH}.

\begin{proof}
To save space we will write $F\coloneqq C_0(\widehat{A},H)$. Let $t\in \TC_A(F)$ be a trace-class operator. Using the weak polar decomposition in the $C^*$-algebra $\Adjointable_A(F)$ we write $t=rs$, where $s=|t|^{1/2}\in \Adjointable_A(F)$, and $r\in \Adjointable_A(F)$ satisfies $r^*r=|t|$. (See \cite[1.4.5]{Pedersen}, where we set $x=t$, $a=t^*t$, and $\alpha=\frac{1}{4}$.)  Since $t\in \TC_A(F)$, both $r^*r$ and $ss^*$ lie in $\TC_A(F)$. 

Let $(\epsilon_i)$ be an orthonormal basis for $H$, and let $(\widetilde{\epsilon_i})$ be the corresponding frame of multipliers for $F$ (see Example \ref{example:free-module-frame}). For each $i$ we define $\eta_i\coloneqq r \widetilde{\epsilon_i}$. A priori $\eta_i$ lies in  $\Multiplier(F)$, but we will now show that in fact $\eta_i\in F$. Since $r^*r\in \TC_A(F)$, the series $\trace(r^*r)=\sum_i \langle \eta_i\, |\, \eta_i\rangle$ converges to an element of $A$. This is a series of positive bounded functions on $\widehat{A}$, and its sum vanishes at infinity, so the same must be true of each of the summands: that is, $\langle \eta_i\, |\, \eta_i\rangle\in A$ for each $i$. Since this inner product is the operator $\eta_i^*\eta_i$, and since this operator lies in the ideal $A=\Compact_A(A)$ of $\Adjointable_A(A)$, we conclude that indeed $\eta_i\in \Compact_A(A,F)\cong F$.

We next show that the sum $\sum_{i} \ket{\eta_i}\bra{\widetilde{\epsilon_i}}$  converges to $r$ in the operator norm on $\Adjointable_A(F)$. To see that the series converges, note that $\langle \widetilde{\epsilon_i}\ |\ \widetilde{\epsilon_j}\rangle=\delta_{i,j}$, and so
\begin{equation}\label{eq:Cauchy}
\big\|\textstyle\sum_{i=n}^{n+m} \ket{\eta_i}\bra{\widetilde{\epsilon_i}}\big\|^2_{\Adjointable_A(F)} 
= \big\| \textstyle \sum_{i=n}^{n+m}\ket{\eta_i}\bra{\eta_i} \big\|_{\Adjointable_A(F)} 
\leq \big\| \textstyle \sum_{i=n}^{n+m} \langle \eta_i\, |\, \eta_i \rangle \big\|_A,
\end{equation}
where we once again used the fact that the trace dominates the operator norm for positive operators on each Hilbert space $F_x$. We observed above that the series
\(
\sum_{i} \langle \eta_i\, |\, \eta_i\rangle
\)
converges in norm in $A$, and so the estimate \eqref{eq:Cauchy} shows that $\sum_{i} \ket{\eta_i}\bra{\widetilde{\epsilon_i}}$ converges too. The last assertion in Theorem \ref{thm:frame-properties} implies that the strong-operator limit of this series is $r$, and so the norm limit is also $r$.

Since $ss^*\in \TC_A(F)$, the same argument as above shows that we can write $s$ as a norm-convergent series $\sum_{i} \ket{\widetilde{\epsilon_i}}\bra{\xi_i}$, where each $\xi_i$ lies in $F$, and where $\sum_{i}\langle \xi_i\, |\ \xi_i\rangle$ converges in norm in $A$. Now it is clear from the definition of the Haagerup norm that the series
\(
u = \sum_{i} \bra{\xi_i}\otimes \ket{\eta_i}
\)
converges in $F^*\otimes^{\h} F$, and the sum $u$ of this series satisfies
\(
\phi_F(u)=\sum_{i} \ket{\eta_i}\bra{\xi_i} = rs = t.
\)
Thus $t\in \image \phi_F$.
\end{proof}

\begin{lemma}\label{lem:image-phi-2}
Let $H$ be a separable Hilbert space. For each $u\in C_0(\widehat{A},H)^*\otimes^{\h} C_0(\widehat{A},H)$ the function $x\mapsto \phi_{C_0(\widehat{A},H)}(u)_x$ is contained in $C_0(\widehat{A},\TC(H))$. 
\end{lemma}

\begin{proof}
It is straightforward to check that the map
\[
\psi: C_0(\widehat{A},H)^*\otimes C_0(\widehat{A},H)\to C_0(\widehat{A}, H^*\otimes^{\h} H),\quad \psi(\bra{\xi}\otimes \ket{\eta})(x)=\bra{\xi(x)}\otimes \ket{\eta(x)}
\]
is contractive with respect to the Haagerup norm, and thus extends to the Haagerup tensor product. The diagram
\[
\xymatrix@C=50pt{
C_0(\widehat{A},H)^* \otimes^{\h} C_0(\widehat{A},H) \ar[r]^-{\phi_{C_0(\widehat{A},H)}} \ar[d]_-{\psi} & \Compact_A(C_0(\widehat{A},H)) \ar[d]^-{k\mapsto (x\mapsto k_x)} \\
C_0(\widehat{A}, H^*\otimes^{\h} H) \ar[r]^-{f\mapsto \phi_H\circ f} & C_0(\widehat{A}, \Compact(H))
}
\]
obviously commutes, and Lemma \ref{lem:phiH} ensures that the map $\phi_H$ is an isometry into $\TC(H)$. 
\end{proof}

\begin{proof}[Proof of Theorem \ref{thm:SvS}]
The characterisation of $\TC_A(C_0(\widehat{A},H))$ given in Theorem \ref{thm:tc-pointwise}(c) immediately gives the inclusion $C_0(\widehat{A},\TC(H))\subseteq \TC_A(C_0(\widehat{A},H))$. The reverse inclusion holds because 
\[
\TC_A(C_0(\widehat{A},H)) \underset{\text{Lemma \ref{lem:image-phi-1}}}{\subseteq} \image \phi_{C_0(\widehat{A},H)} \underset{\text{Lemma \ref{lem:image-phi-2}}}{\subseteq} C_0(\widehat{A},\TC(H)).\qedhere
\]
\end{proof}

\section{Operator-space structure on $\TC_A(F)$}\label{sec:oss}

Let $F$ be a Hilbert module, countably generated by multipliers, over a commutative $C^*$-algebra $A$. We are going to equip the space $\TC_A(F)$ of trace-class operators on $F$ with an  operator-space structure, by embedding $F$ into a free module; and prove that with this operator-space structure $\TC_A(F)$ is completely isometrically isomorphic to a Haagerup tensor product. 

In this section we assume that the reader is conversant with the basic facts about operator spaces, and we appeal to well-known results from that subject more freely than in the previous section; general references are \cite{Pisier}, \cite{ER},  \cite{BLM}, and \cite{Blecher-newapproach}.

Let us briefly recall that if $F$ is a right Hilbert module over a $C^*$-algebra $A$, then we regard $F$ as a right \emph{operator} module over $A$ by embedding  $F$ and $A$ into  the $C^*$-algebra 
\[
\Adjointable_A(F\oplus A) = \begin{bmatrix} \Adjointable_A(F) & \Adjointable_A(A,F) \\ \Adjointable_A(F,A) & \Multiplier(A) \end{bmatrix}
\]
via the maps $\xi\mapsto \left[\begin{smallmatrix}0 & \ket{\xi} \\ 0& 0\end{smallmatrix}\right]$ and $a\mapsto \left[\begin{smallmatrix} 0&0 \\ 0& a\end{smallmatrix}\right]$. 

If $A$ is commutative then $F$ is also a \emph{left} operator module over $A$:  indeed,  for each $a\in A$ and each $\xi\in F$ the map $\xi\mapsto a\xi\coloneqq \xi a$ is an adjointable operator on $F$, and the map $A\to \Adjointable_A(F)$ sending $a\in A$ to this operator $\xi\mapsto a\xi$ is a nondegenerate $*$-homomorphism, by means of which $F$ becomes a left operator $A$-module. Taking adjoints, $F^*$ is also an operator $A$-bimodule: explicitly, $\bra{\xi}a = a\bra{\xi} = \bra{\xi a^*}$ for all $\xi\in F$ and $a\in A$.

The definition of the Haagerup norm  on $F^*\otimes^{\h} F$ (Definition \ref{def:Haagerup-norm}) extends in a natural way to give an operator-space structure on $F^*\otimes^{\h} F$ (see eg \cite[Chapter 9]{ER} or \cite[Chapter 5]{Pisier}.)  The balanced Haagerup tensor product
\[
F^*\otimes^{\h}_A F \coloneqq (F^*\otimes^{\h} F)/\overline{\lspan}\{ \bra{\xi a^*}  \otimes \ket{\eta} - \bra{\xi}\otimes  \ket{\eta a} \ |\ \eta,\xi\in F,\ a\in A\}
\]
is then made into an operator space using the quotient operator-space structure. We noted above that $F^*$ and $F$ are operator $A$-bimodules, and so $F^*\otimes^{\h}_A F$ is also an operator $A$-bimodule. (See \cite[Section 3.4]{BLM} or \cite{Blecher-newapproach} for details about Haagerup tensor products of operator (bi)modules.) The map $\phi_F : F^*\otimes^{\h} F \to \Compact_A(F)$, $\bra{\xi}\otimes \ket{\eta}\mapsto \ket{\eta}\bra{\xi}$ of Lemma \ref{lem:phiF} vanishes on elements of the form $\bra{\xi a^*}\otimes \ket{\eta}-\bra{\xi}\otimes \ket{\eta a})$ , and thus descends to a contractive map $\phi_F:F^*\otimes^{\h}_A F \to \Compact_A(F)$.

We are going to  prove that if $F$ is countably generated by multipliers, then $\phi_F$ is a completely isometric isomorphism $F^*\otimes^{\h}_A F\xrightarrow{\cong} \TC_A(F)$. In order to make sense of this we first need to equip $\TC_A(F)$ with an operator-space structure. When $F=C_0(\widehat{A},H)$ for a separable Hilbert space $H$, there is an obvious way to do this: we have in this case $\TC_A(F)\cong C_0(\widehat{A},\TC(H))$ (Theorem \ref{thm:SvS}), and the right-hand side carries a canonical operator-space structure coming from the identifications $M_n(C_0(\widehat{A},\TC(H)))\cong C_0(\widehat{A}, M_n(\TC(H)))$ and $\TC(H) \cong \CB(\Compact(H),\C)$.

Now we let $F$ be any Hilbert module over $A$, countably generated by multipliers, and we choose a separable Hilbert space $H$ and an adjointable map $\theta: F\to C_0(\widehat{A},H)$ with $\theta^*\theta=\id_F$, as in Theorem \ref{thm:frame-properties}. The map $\Adjointable_A(F)\to \Adjointable_A(C_0(\widehat{A},H))$ defined by $t\mapsto \theta t \theta^*$ is an injective $*$-homomorphism.

\begin{lemma}[cf. {\cite[Theorem 3.5]{SvS}}]\label{lem:Theta}
For each $t\in \Adjointable_A(F)$ we have $t\in \TC_A(F)$ if and only if $\theta t \theta^*\in \TC_A(C_0(\widehat{A},H))$. If these inclusions hold then $\trace(t) = \trace(\theta t \theta^*)$. 
\end{lemma}

\begin{proof}
The proof is the same as in \cite{SvS}: choose an orthonormal basis $(\epsilon_i)$ for $H$, and let $(\widetilde{\epsilon_i})$ be the corresponding frame of multipliers for $C_0(\widehat{A},H)$. The fact that $\theta^*\theta=\id_F$ ensures that the sequence ($\theta^*\widetilde{\epsilon_i}$) is a frame of multipliers for $F$, and that the series defining $\trace_{\theta^*\widetilde{\epsilon}}(t)$ and $\trace_{\widetilde{\epsilon}}(\theta t \theta^*)$ are identical term by term.
\end{proof}

Lemma \ref{lem:Theta} implies, firstly, that $\TC_A(F)$ is a linear subspace of $\Compact_A(F)$ (since Theorem \ref{thm:SvS} ensures that this is true when $F=C_0(\widehat{A},H)$). Moreover we can pull back the canonical operator-space structure on $\TC_A(C_0(\widehat{A},H))$ along the embedding $t\mapsto \theta t \theta^*$ to obtain an operator-space structure on $\TC_A(F)$, with the norm being $\|t\|_{\Adjointable_A(F)} = \trace(|t|)$. 

Now that we have equipped $\TC_A(F)$ with an operator-space structure, we can formulate and prove our main result:

\begin{theorem}\label{thm:ci}
Let $A$ be a commutative $C^*$-algebra, and let $F$ be a Hilbert $A$-module that is countably generated by multipliers. The map
\[
\phi_F : F^*\otimes^{\h}_A F \to \TC_A(F),\qquad \phi_F(\bra{\xi}\otimes \ket{\eta}) = \ket{\eta}\bra{\xi}
\]
is a completely isometric isomorphism, with respect to the operator-space structure induced on $\TC_A(F)$ by any adjointable isometry $\theta: F \to C_0(\widehat{A},H)$.
\end{theorem}

The next lemma proves Theorem \ref{thm:ci} in the case $F=C_0(\widehat{A},H)$. The proof of the general case is given below the proof  of the lemma.

\begin{lemma}\label{lem:free-module-ci}
If $A$ is a commutative $C^*$-algebra, and $H$ is a separable Hilbert space, then the map $\phi_{C_0(\widehat{A},H)}$ gives a completely isometric isomorphism of operator spaces
\[
C_0(\widehat{A},H)^*\otimes^{\h}_A C_0(\widehat{A},H) \xrightarrow{\cong} \TC_A(C_0(\widehat{A},H)).
\]
\end{lemma}

\begin{proof}
We begin by making the completely isometric identifications
\[
C_0(\widehat{A},H)\cong A\otimes^{\min} H,\quad C_0(\widehat{A},H)^*\cong A\otimes^{\min} H^*, \quad \text{and} \quad 
\]
\[
\TC_A(C_0(\widehat{A},H))\cong C_0(\widehat{A},\TC(H))\cong A\otimes^{\min} \TC(H)\cong A\otimes^{\min} (H^*\otimes^{\h} H),
\]
where $\otimes^{\min}$ is the minimal tensor product of operator spaces; cf.  \cite[Proposition 1.5.3]{BLM}. Our assertion, then, is that  the map
\begin{equation}\label{eq:ci1}
\begin{aligned}
 (A\otimes^{\min} H^*) \otimes_A^{\h} (A\otimes^{\min} H) &\to A\otimes^{\min} (H^*\otimes^{\h} H), \\
 (a_1\otimes \bra{\xi})\otimes(a_2 \otimes \ket{\eta}) &\mapsto a_1 a_2\otimes (\bra{\xi}\otimes \ket{\eta})
\end{aligned}
\end{equation}
is a completely isometric isomorphism. This map obviously has dense image, so we are left to prove that it is completely isometric.

To prove this, embed $H$ and $H^*$ completely isometrically into a unital $C^*$-algebra $B$ (for example, the $C^*$-algebra of bounded operators on $H\oplus \C$), and embed $A$ into its minimal unitisation $\widetilde{A}$. Since $\otimes^{\min}$ and $\otimes^{\h}_A$ are injective, it will be enough to prove that the map
\begin{equation}\label{eq:ci2}
\begin{aligned}
(\widetilde{A}\otimes^{\min} B) \otimes^{\h}_{\widetilde{A}} (\widetilde{A}\otimes^{\min} B) & \to \widetilde{A}\otimes^{\min} (B\otimes^{\h} B), \\
(a_1\otimes b_1)\otimes (a_2 \otimes b_2) &\mapsto a_1 a_2\otimes (b_1\otimes b_2) 
\end{aligned}
\end{equation}
is completely isometric. Now, if $C$ is a unital $C^*$-subalgebra of a unital $C^*$-algebra $D$, then the Haagerup tensor product $D\otimes^{\h}_C D$ embeds completely isometrically into the amalgamated free product $D\star_C D$, via the map $d_1\otimes d_2\mapsto \delta_1(d_1)\delta_2(d_2)$ (where the $\delta_i$ are the canonical maps from $D$ into the free product). This fact was pointed out in \cite{Ozawa},  generalising earlier results from \cite{CES} and \cite{Pisier-Kirchberg}. Applying this fact to $D=\widetilde{A}\otimes^{\min} B$ and $C=\widetilde{A}\otimes^{\min} \C 1_B$, we find that in order to prove that \eqref{eq:ci2} is a complete isometry it will suffice to prove that the $*$-homomorphism
\begin{equation}\label{eq:free-products}
(\widetilde{A}\otimes^{\min} B) \star_{\widetilde{A}} (\widetilde{A}\otimes^{\min} B)  \to \widetilde{A}\otimes^{\min} (B\star_{\C} B)
\end{equation}
induced by the $*$-homomorphisms $\id\otimes \beta_i : \widetilde{A}\otimes^{\min} B \to \widetilde{A}\otimes^{\min}(B\star_{\C} B)$ is injective. A comparison of the universal properties on each side (recalling that $A$ is nuclear) shows that \eqref{eq:free-products} is in fact an isomorphism, and this completes the proof that \eqref{eq:ci1} is completely isometric.
\end{proof}

\begin{proof}[Proof of Theorem \ref{thm:ci}]
Choose an adjointable isometry $\theta : F\to C_0(\widehat{A},H)$. The diagram 
\[
\xymatrix@C=60pt{
F^*\otimes^{\h}_A F \ar[r]^-{\phi_F} \ar[d]_-{\bra{\xi}\otimes \ket{\eta}\mapsto \bra{\theta\xi}\otimes \ket{\theta\eta}} & \Compact_A(F) \ar[d]^-{t\mapsto \theta t \theta^*} \\
C_0(\widehat{A},H)^*\otimes^{\h}_A C_0(\widehat{A},H) \ar[r]^-{\phi_{C_0(\widehat{A},H)}} & \Compact_A(C_0(\widehat{A},H))
}
\]
is easily seen to commute. The left-hand vertical arrow is a complete isometry because $\theta$ is a complete isometry, and because $\otimes^{\h}_A$ is injective. Lemma \ref{lem:free-module-ci} implies that the image of   $\phi_{C_0(\widehat{A},H)}$ is precisely $\TC_A(C_0(\widehat{A},H))$, and so Lemma \ref{lem:Theta} ensures that the image of $\phi_F$ is contained in $\TC_A(F)$. Moreover, Lemma \ref{lem:free-module-ci} implies that $\phi_{C_0(\widehat{A},H)}$ is a complete isometry into $\TC_A(C_0(\widehat{A},H))$, and the operator-space structure on $\TC_A(F)$ is defined so that the right-hand vertical arrow is a complete isometry $\TC_A(F)\to \TC_A(C_0(\widehat{A},H))$, and hence $\phi_F:F^*\otimes^{\h}_A F\to \TC_A(F)$ is a complete isometry.

We are left to show that $\image\phi_F = \TC_A(F)$, which we do by reversing the vertical arrows in the above diagram, yielding the diagram
\[
\xymatrix@C=60pt{
F^*\otimes^{\h}_A F \ar[r]^-{\phi_F}  & \TC_A(F)  \\
C_0(\widehat{A},H)^*\otimes^{\h}_A C_0(\widehat{A},H) \ar[r]^-{\phi_{C_0(\widehat{A},H)}} \ar[u]^-{\bra{\xi}\otimes \ket{\eta}\mapsto \bra{\theta^*\xi}\otimes \ket{\theta^*\eta}} & \TC_A(C_0(\widehat{A},H)) \ar[u]_-{t\mapsto \theta^* t \theta}
}
\]
which still commutes. Both vertical arrows in this diagram are surjective (obvious on the left, and Lemma \ref{lem:Theta} on the right), and $\phi_{C_0(\widehat{A},H)}$ is also surjective (Lemma \ref{lem:free-module-ci}), and so $\phi_F$ is surjective onto $\TC_A(F)$.
\end{proof}

\section{An example from harmonic analysis}\label{sec:example}

Let $G$ be a locally compact group. For each closed subgroup $H$ of $G$ we denote by $I_H$ the $C^*(G)$-$C^*(H)$-bimodule constructed by Rieffel in \cite{Rieffel-induced} to represent the functor of unitary induction of representations from $H$ to $G$. Referring to \cite{Rieffel-induced} or \cite[Appendix C]{Raeburn-Williams} for the details, we recall briefly that  $I_H$ is the completion of $C_c(G)$ in the norm induced by a  $C^*(H)$-valued inner product defined by convolving functions on $G$ and then restricting to $H$; and that the actions of $H$ and of $G$ on $C_c(G)$ by right and left translation, respectively, turn $I_H$ into a right Hilbert $C^*(H)$-module, equipped with a $*$-homomorphism $C^*(G)\to \Adjointable_{C^*(H)}(I_H)$. In particular, $I_H$ is an operator $C^*(G)$-$C^*(H)$-bimodule. The most important property of this bimodule is that if $V$ is a Hilbert space equipped with a unitary representation $\chi$ of $H$, then $I_H\otimes^{\h}_{C^*(H)} V$ is isomorphic to the induced unitary representation $\Ind_H^G\chi$ of $G$.

The Haagerup tensor product 
\[
A_H(G)\coloneqq I_H^* \otimes^{\h}_{C^*(G)} I_H
\] 
is an operator $C^*(H)$-bimodule. In \cite{Crisp-descent} it was shown that if $H$ is cocompact in $G$ then $A_H(G)$ carries an algebraic structure---it is a coalgebra over $C^*(H)$, with respect to the Haagerup tensor product---and that this coalgebra is  Morita equivalent, in a suitable sense, to $C^*(G)$. Even in cases where $H$ is not cocompact in $G$, computing the operator bimodule $A_H(G)$ appears to be an interesting problem. In the two extreme cases, we have $A_G(G)=C^*(G)$, the group $C^*$-algebra; and $A_{\{1\}}(G)=A(G)$, Eymard's Fourier algebra \cite{Eymard-Fourier}. The family of bimodules $A_H(G)$ can be seen as interpolating between these two extremes.

Here we will consider the case where $G$ is a compact, second-countable, abelian group. The Pontryagin dual $\widehat{G}$ is then a countable discrete group, and the Fourier transform gives completely isometric isomorphisms $A_G(G)\cong C_0(\widehat{G})$ and $A_{\{1\}}(G)\cong \ell^1(\widehat{G})$. We will use Theorem \ref{thm:ci} to describe $A_H(G)$, as a space of functions on $\widehat{G}$, for an arbitrary closed subgroup $H$ of $G$. 

\begin{definition}
Let $H$ be a closed subgroup of a compact, second-countable, abelian group $G$. For each character $\chi\in \widehat{H}$ we define
\[
\widehat{G}_\chi \coloneqq \left\{\phi\in \widehat{G}\ \middle|\ \phi\restrict_H = \chi\right\}.
\]
We let $\ell_H(\widehat{G})$ denote the completion of $C_c(\widehat{G})$ in the norm
\[
\| f\|_{\ell_H} \coloneqq \sup_{\chi\in \widehat{H}} \Big(\sum_{\phi\in \widehat{G}_\chi } |f(\phi)|\Big).
\]
There is an obvious isometric identification
\begin{equation}\label{eq:l_H-iso}
\ell_H(\widehat{G}) \cong \left\{ f\in C_0(\widehat{H}, \ell^1(\widehat{G}))\ \middle|\ f(\chi)\in \ell^1(\widehat{G}_\chi)\ \text{ for all }\chi\in \widehat{H} \right\},
\end{equation}
and $C_0(\widehat{H},\ell^1(\widehat{G}))$ is an operator space via the identifications $M_n\left(C_0(\widehat{H},\ell^1(\widehat{G}))\right) \cong C_0\left(\widehat{H}, M_n(\ell^1(\widehat{G}))\right)$ and $\ell^1(\widehat{G})\cong \CB(C_0(\widehat{G}),\C)$. We equip $\ell_H(\widehat{G})$ with the operator-space structure that it inherits as a closed subspace of $C_0(\widehat{H},\ell^1(\widehat{G}))$.
\end{definition}

\begin{theorem}
Let $H$ be a closed subgroup of a compact, second-countable, abelian group $G$. There is a completely isometric isomorphism $A_H(G)\cong \ell_H(\widehat{G})$.
\end{theorem}

\begin{proof}
The fact that tensor product with $I_H$ implements unitary induction of representations, together with Frobenius reciprocity, implies that for each $\chi\in \widehat{H}$ we have 
\[
(I_H)_{\chi} \cong \Ind_H^G \chi \cong \ell^2(\widehat{G}_\chi),
\] 
the direct sum of the one-dimensional $G$-representations $\phi\in \widehat{G}_\chi$. We thus have a unitary isomorphism of Hilbert modules over $C^*(H)\cong C_0(\widehat{H})$,
\begin{equation}\label{eq:I_H-Fourier}
I_H \cong \left\{ \xi \in C_0(\widehat{H},\ell^2(\widehat{G}))\ \middle|\ \xi(\chi)\in \ell^2(\widehat{G}_\chi )\ \text{for each}\ \chi\in \widehat{H} \right\}.
\end{equation}
The action of $C^*(G)$ on the right-hand side of \eqref{eq:I_H-Fourier} is the fibrewise action on each of the $G$-representations $\ell^2(\widehat{G}_\chi)$.

Since $G$ is second-countable the Hilbert space $\ell^2(\widehat{G})$ is separable, and so the picture of $I_H$ given in \eqref{eq:I_H-Fourier} shows (thanks to Theorem \ref{thm:frame-properties}) that $I_H$ is countably generated by multipliers. Combining the identification \eqref{eq:I_H-Fourier} with Theorem \ref{thm:SvS} and Lemma \ref{lem:Theta} gives a competely isometric isomorphism between the space of trace-class operators $\TC_{C^*(H)}(I_H)$ and the space
\begin{equation}\label{eq:TC-Fourier}
\mathcal{L}_H\coloneqq  \left\{ t\in C_0\left(\widehat{H},\TC(\ell^2(\widehat{G}))\right)\ \middle|\ t(\chi)\in \TC(\ell^2(\widehat{G}_\chi ))\ \text{for each}\ \chi\in \widehat{H} \right\}.
\end{equation}
For each $\chi\in \widehat{H}$ we have a completely isometric embedding $\ell^1(\widehat{G}_\chi)\into \TC(\ell^2(\widehat{G}_\chi))$, sending each $\ell^1$-function to the associated pointwise-multiplication operator. In conjunction with \eqref{eq:l_H-iso} and \eqref{eq:TC-Fourier}, this observation shows that we have a completely isometric embedding $\ell_H(\widehat{G})\into \mathcal{L}_H$.

Theorem \ref{thm:ci} gives a completely isometric isomorphism
\begin{equation}\label{eq:I_H-TC}
I_H^*\otimes^{\h}_{C^*(H)} I_H \cong \TC_{C^*(H)}(I_H) \cong \mathcal{L}_H.
\end{equation}
The space $A_H(G)=I_H^*\otimes^{\h}_{C^*(G)} I_H$ is the quotient of $I_H^*\otimes^{\h}_{C^*(H)} I_H$ by the subspace 
\[
\overline{\lspan}\{ \bra{\xi}a\otimes \ket{\eta}-\bra{\xi}\otimes a\ket{\eta}\ |\ \xi,\eta\in I_H,\ a\in C^*(G)\}.
\]
(Here it is important that the subgroup $H$ is central in $G$, so that the actions of $C^*(H)$ on $I_H$ by left and right translations coincide.) A standard approximation argument shows that the latter subspace is equal to 
\[
\overline{\lspan}\{\bra{\xi}g\otimes \ket{\eta} - \bra{\xi}\otimes g\ket{\eta}\ |\ \xi,\eta\in I_H,\ g\in G\},
\] 
which is in turn equal to
\[
\overline{\lspan}\{ \bra{\xi}g\otimes g^{-1}\ket{\eta} - \bra{\xi}\otimes \ket{\eta}\ |\ \xi,\eta\in I_H,\ g\in G\}.
\]
The action of $G$ on $I_H^*\otimes^{\h}_{C^*(H)} I_H$ given by $g:\bra{\xi}\otimes \ket{\eta}\mapsto \bra{\xi}g\otimes g^{-1}\ket{\eta}$ corresponds, under the isomorphism \eqref{eq:I_H-TC}, to the fibrewise (over $\widehat{H}$) action of $G$ on $\TC(\ell^2(\widehat{G}))$ by conjugation: that is, for each function $t\in \mathcal{L}_H$,  each $g\in G$, and each $\chi\in \widehat{H}$, we have
\[
(g\cdot t)(\chi) = \hat{\lambda}_g^* t(\chi) \hat{\lambda}_g
\]
where $\hat{\lambda}:G\to \Unitary(\ell^2(\widehat{G}))$ denotes the Fourier transform of the regular representation. This action of $G$ on $\mathcal{L}_H$ is by complete isometries, because over each $\chi\in\widehat{H}$, each $g\in G$ acts as the dual of a $*$-automorphism of $\Compact(\ell^2(\widehat{G}_\chi))$.

Since $G$ is a compact group,  integration over $G$ gives a completely isometric isomorphism between the coinvariant space
\[
\mathcal{L}_H/ \overline{\lspan}\{g\cdot t - t\ |\ t\in \mathcal{L}_H,\ g\in G\}
\]
and the space of invariants
\[
\mathcal{L}_H^G = \{ t\in \mathcal{L}_H\ |\ g\cdot t = t\ \text{for all}\ g\in G\}.
\]
For $t\in \mathcal{L}_H$ we have $g\cdot t = t$ for all $g\in G$ if and only if, for each $\chi\in \widehat{H}$, the trace-class operator $t(\chi)\in \TC(\ell^2(\widehat{G}_\chi))$ commutes with the regular representation $\hat{\lambda}$---which is to say, if and only if $t(\chi)$ is the operator of pointwise multiplication by some $\ell^1$-function on $\widehat{G}_\chi$. 
Thus  $\mathcal{L}_H^G$  is precisely the subspace of $\mathcal{L}_H$ that we previously identified, completely isometrically, with $\ell_H(\widehat{G})$.

In summary, we have produced a chain of completely isometric isomorphisms 
\begin{align*}
A_H(G) & \cong \left(I_H^*\otimes^{\h}_{C^*(H)} I_H\right) / \overline{\lspan}\{ \bra{\xi}a\otimes \ket{\eta}-\bra{\xi}\otimes a\ket{\eta}\ |\ \xi,\eta\in I_H,\ a\in C^*(G)\} \\
& \cong \mathcal{L}_H/ \overline{\lspan}\{g\cdot t - t\ |\ t\in \mathcal{L}_H,\ g\in G\} 
 \cong \mathcal{L}_H^G  \cong \ell_H(\widehat{G}).\qedhere
\end{align*}
\end{proof}

\section*{Acknowledgements}
We gratefully acknowledge support from the Royal Irish Academy, the American Mathematical Society, and the Simons Foundation. We thank the referee for their careful reading and helpful suggestions.

\bibliographystyle{alpha}
\bibliography{traceclass.bib}
 
 \end{document}